\documentclass[oneside]{amsart}
\usepackage{amsmath,amssymb,amsfonts,amsthm}
\usepackage{color}
\usepackage{graphicx}
\usepackage{geometry} % Paolo: authos,affiliation
\usepackage{amscd}
\usepackage{enumitem}

\title{On the existence of parallel one forms}

\author{L\'aszl\'o Kozma and   S. G. Elgendi }

\address{L\'aszl\'o Kozma, Department of Geometry, Institute of Mathematics, University of Debrecen,
	H-4002 Debrecen, P. O. Box 400,  Hungary} \email{kozma@unideb.hu}
\urladdr{www.math.unideb.hu/en/dr-laszlo-kozma}

\address{Salah G. Elgendi, Department of Mathematics, Faculty of Science, Benha
	University, Egypt} \email{salah.ali@fsci.bu.edu.eg, \, salahelgendi@yahoo.com}
\urladdr{http://www.bu.edu.eg/staff/salahali7}

\keywords{Parallel one form;  Killing vector field; Spray; metrizability freedom.}

\subjclass[2020]{53C60, 53B40, 58B20.}
\thanks{}

\def\blue#1{\textcolor[rgb]{0.0,0.0,1.0}{#1}}

\newcommand{\T}{{\mathcal T}}
\newcommand{\C}{{\mathcal C}}

\newcommand{\Real}{\mathbb R}

\newcommand{\tm}{\T M}

\newcommand{\N}{\mathcal{N}}
\newcommand{\TM}{\mathcal T\hspace{-1pt}M}

\def\hol{{\mathcal H_{S}}}
\def\H{{{\mathcal D}_{\mathcal H}}}

\def\pa{\partial}
\def\paa{\dot{\partial}}

\def\+{\!+\!}

\def\={\!=\!}
\def\<{\!<\!}
\def\>{\!>\!}

 \let\oldmarginpar\marginpar
\renewcommand\marginpar[1]{\oldmarginpar[\raggedleft\footnotesize #1]%
  {\blue{\raggedright \footnotesize \fbox{
      \begin{minipage}{1.0\linewidth}
        #1
      \end{minipage}
}}}}

\numberwithin{equation}{section} %% Comment out for sequentially-numbered
\numberwithin{figure}{section} %% Comment out for sequentially-numbered

\theoremstyle{plain}
\newtheorem*{theorem*}{Theorem}
\newtheorem{theorem}{Theorem}[section]
\newtheorem{lemma}[theorem]{Lemma}
\newtheorem{proposition}[theorem]{Proposition}
\newtheorem{corollary}[theorem]{Corollary}
\theoremstyle{definition}
\newtheorem{definition}[theorem]{Definition}

\theoremstyle{remark}
\newtheorem{example}{Example}
\newtheorem{remark}[theorem]{Remark}
\newtheorem*{acknowledgement*}{Acknowledgement}

\begin{document}

\maketitle

\begin{abstract}
  In this paper,  using the Finslerian settings, we study  the existence of parallel one forms (or, equivalently parallel vector fields) on a Riemannian manifold. We show that a parallel one form on a Riemannian manifold $M$ is  a holonomy invariant function on the tangent bundle $TM$ with respect to the geodesic spray.  We prove that   if  the metrizability freedom of the geodesic spray of $(M,F)$ is  $1$, then the $(M,F)$ does not admit a  parallel one form.  We investigate   a sufficient  condition  on a Riemannian manifold to admit a parallel one form.  As by-product, we relate the existence of a proper  affine Killing vector field by the metrizability freedom.  We establish sufficient   conditions for the existence of a parallel one form on a Finsler manifold.   By counter-examples, we show that if the metrizability freedom is greater than 1, then the manifold (Riemannian or Finslerian) does not necessarily admit a parallel one form.  Various special cases and examples are studied and discussed. 
 \end{abstract}

\section{Introduction}

The parallel vector fields or parallel one forms have many applications not only in Riemannian and Finslerian geometries but also in physics especially in general relativity.  For example, if  the metric $g$ satisfies Einstein equations  and admits a non-trivial parallel vector field, then the energy- momentum tensor is identically zero (cf. \cite{Mahara}).

Let $M$ be a Riemannian manifold. Since the Levi-Civita connection is metrical, the associated  one form of a parallel vector field on $M$ is parallel and vice versa. So one can say that the concepts of parallel one form and parallel vector field on a Riemannian manifold are equivalent.

In the theory of $(\alpha,\beta)$-metrics, if $\beta$ is parallel with respect to Levi-Civita connection, then many interesting results can be obtained. For example,  $(\alpha,\beta)$-metric and the Riemannain metric $\alpha$ have the same geodesic spray, which will be quadratic in that case, and hence the $(\alpha,\beta)$-metric is Berwaldian. Moreover, the Cartan connection of $(\alpha,\beta)$ and the Levi-Civita connection coincide if and only if $\beta$ is parallel (cf. \cite{Shibata} ).

The existence of parallel vector fields on a Riemannnian manifold has been studied by many authors, for example we refer to \cite{Percell}. Depending on these studies, one can see that the existence of such vector fields is equivalent to some topological restrictions on the manifold.

\bigskip

In this paper, we use the Finslerian setting to study the existence of a parallel  form  (one form) on a Riemannian manifold $M$ and a Finsler manifold as well.
The Finslerian settings, in this topic, are much easier and  interesting than the topological ones. Moreover, it gives some interesting geometric properties and many examples can be considered and studied.
 We show that a parallel one form is a holonomy invariant function on $TM$ with respect to the geodesic spray. If a Riemannian manifold admits a parallel one form then the metrizability freedom of the geodesic spray is greater than one.    Or  equivalently, if the metrizability freedom of the geodesic spray of a Riemannian manifold is $1$ then the manifold does not admit a parallel one form.
  We prove that a sufficient  condition  for a Riemannian manifold $(M,F)$ to admit a parallel one form  is
 $$  R^\mu _{hij}=0$$
 for some indices $\mu$, where $R^\mu _{hij}$ are the components of the Riemannian curvature.
 As an application, we show that if a Riemannian  manifold admits a proper affine Killing vector field then the  metrizability  freedom of its geodesic spray   is greater than $1$. Also, when  $M$ is a two dimensional Riemannian manifold  with the geodesic spray $S$ of non-vanishing curvature, then the $M$ does not admit a parallel one form.

For the Finslerian case, we provide  sufficient conditions for a Finsler manifold $(M,F)$ to admit a parallel one form, namely, 
$$R^\mu _{ij}=0, \quad G^\mu_{ijk}=0\ (\equiv G^\mu_{jk}=G^\mu_{jk}(x)).$$
for some indices $\mu$, where $R^\mu _{ij}$ are the components of the curvature of the geodesic spray and $G^\mu_{ijk} $ are the components of the Berwald curvature. By counter-examples, we show that if the metrizability freedom is greater than 1, then the manifold (Riemannian or Finslerian) does not necessarily admit a parallel one form (cf. Examples 3 and 4).

\section{Preliminaries}

Let $M$ be an $n$-dimensional manifold and $(TM,\pi_M,M)$ be its tangent bundle
and $(\T M,\pi,M)$ the subbundle of nonzero tangent vectors.  We denote by
$(x^i) $ local coordinates on the base manifold $M$ and by $(x^i, y^i)$ the
induced coordinates on $TM$.  The vector $1$-form $J$ on $TM$ defined,
locally, by $J = \frac{\partial}{\partial y^i} \otimes dx^i$ is called the
natural almost-tangent structure of $T M$. The vertical vector field
$\C=y^i\frac{\partial}{\partial y^i}$ on $TM$ is called the canonical or the
Liouville vector field.

A vector field $S\in \mathfrak{X}(\T M)$ is called a spray if $JS = \C$ and
$[\C, S] = S$. Locally, a spray can be expressed as follows
\begin{equation}
  \label{eq:spray}
  S = y^i \frac{\partial}{\partial x^i} - 2G^i\frac{\partial}{\partial y^i},
\end{equation}
where the \emph{spray coefficients} $G^i=G^i(x,y)$ are $2$-homogeneous
functions in   $y$.

A nonlinear connection is defined by an $n$-dimensional distribution $H : u \in \tm \rightarrow H_u\subset T_u(\tm)$ that is supplementary to the vertical distribution, which means that for all $u \in \tm$, we have
$T_u(\tm) = H_u(\tm) \oplus V_u(\tm).$

Every spray S induces a canonical nonlinear connection through the corresponding horizontal and vertical projectors,
\begin{equation}
  \label{projectors}
    h=\frac{1}{2}  (Id + [J,S]), \,\,\,\,            v=\frac{1}{2}(Id - [J,S])
\end{equation}
Equivalently, the canonical nonlinear connection induced by a spray can be expressed in terms of an almost product structure  $\Gamma = [J,S] = h - v$. With respect to the induced nonlinear connection, a spray $S$ is horizontal, which means that $S = hS$. Locally, the two projectors $h$ and $v$ can be expressed as follows
$$h=\frac{\delta}{\delta x^i}\otimes dx^i, \quad\quad v=\frac{\partial}{\partial y^i}\otimes \delta y^i,$$
$$\frac{\delta}{\delta x^i}=\frac{\partial}{\partial x^i}-G^j_i(x,y)\frac{\partial}{\partial y^j},\quad \delta y^i=dy^i+G^j_i(x,y)dx^i, \quad G^j_i(x,y)=\frac{\partial G^j}{\partial y^i}.$$

The Nijenhuis torsion of $h$ measuring the integrability of the horizontal
distribution
\begin{displaymath}
  R=\frac{1}{2}[h,h]=\frac{1}{2}R^i_{jk}\frac{\partial}{\partial
    y^i}\otimes dx^j \wedge dx^k,
\end{displaymath}
\begin{equation}
\label{R12}
R^i_{jk} =
  \frac{\delta
    G^i_j}{\delta x^k} - \frac{\delta G^i_k}{\delta x^j}
\end{equation}
is called the curvature of $S$.

The coefficients of Berwald connection aregiven by
$$G^h_{ij}:=\frac{\partial G^h_i}{\partial y^j}.$$
For a Riemannian manifold $M$ with geodesic spray $S$, the Levi-Civita connection coincides with the Berwald connection.

 Also, the h-curvature tensor of Berwald connection is given by
\begin{equation}
\label{R13}
R^h_{ijk}=\frac{ \delta G^h_{ij}}{\delta x^k}-\frac{ \delta G^h_{ik}}{\delta x^j}+G^h_{mk}G^m_{ij}-G^h_{mj}G^m_{ik}.
\end{equation}
The curvature of the geodesic spray and the curvature tensor are related by
$$R^h_{jk}=y^i R^h_{ijk}.$$

 \begin{definition}[\cite{MZ_ELQ}]
  The \emph{holonomy distribution}, denoted by $\H$, of a given spray $S$ is the distribution on
  $TM$ generated by the horizontal vector fields and their successive
  Lie-brackets, namely,
  \begin{equation}
    \label{eq:14}
    \H := \Bigl\langle \mathfrak X^h(TM)  \Bigr \rangle_{Lie}  \!
    =\Big\{[X_1,[\dots [X_{m-1},X_m]...]] \ \big| \ X_i
    \in \mathfrak X^h(TM) \Big\}
  \end{equation}
  where $\mathfrak X^h(TM)$ is the module of the horizontal vector fields.
\end{definition}
The parallel translation of a vector along a curve  is defined through the horizontal lift as follows:
 \begin{definition}
 Let $\gamma:[0,1] \rightarrow M$ be a curve on $M$ such that $\gamma(0)=p$ and $\gamma(1)=q$. Let $\gamma^{h}(0)=v, \gamma^{h}(1)=w$ where $\gamma^{h}$ is the horizontal lift of the curve $\gamma$ on $TM$, that is, $\pi \circ \gamma^{h}=\gamma$, $\dot{\gamma}^{h}(t) \in H_{{\gamma}^{h}(t)}TM$. That is, the parallel translation $\tau: T_{p} M \rightarrow T_{q} M$ along $\gamma$ is   $\tau(v)=w$.
\end{definition}

\begin{definition}
  A function $E \in C^\infty(TM)$ is called \emph{holonomy
    invariant} with respect to a spray $S$, if it is invariant with respect to the parallel translation, that
  is, for any $v\in TM$ and for any parallel translation $\tau$ we have
  $E(\tau(v))=E(v)$.
  Therefore, $E \in C^\infty(TM)$ is a
  holonomy invariant function if and only if we have $\mathcal{L}_X E=0$, $X \in
  \H $ that is
  \begin{equation}
  	\label{eq:hol}
  	\mathcal H_{S} =
  	\left\{
  	E \in C^\infty(\TM) \ | \ \mathcal{L}_X E=0, \ X \in \H
  	\right\},
  \end{equation}
where $\mathcal H_{S}$ is the set of holonomy invariant functions with respect to $S$.
\end{definition}

\begin{definition}
  A given spray $S$ on a manifold $M$ is called \emph{Finsler metrizable} if there
  exists a Finsler function $F$ such that the geodesic spray of the Finsler
  manifold $(M,F)$ is $S$. So one can can say that $S$ is metrizable if $\hol$ contains a $1$-homogeneous regular element.
\end{definition}

   \begin{definition}\cite{Mu-Elgendi}
  \label{def:v}
   Let $S$ be a metrizable spray, then its \emph{metrizability freedom} is $\mu_S (\in \mathbb N)$ where $\mu_S={rank}\, (\hol)$.  If $S$ is
  non-metrizable, then  $\mu_S=0$.
\end{definition}
The metrizability freedom $\mu_S={rank}\, (\hol)$ means that $\hol$ is
locally generated by  $\mu_S$ functionally independent  elements.  That is, if the metrizability freedom is greater than one then we have essential different holonomy invariant functions and homogeneous of degree one.

\section{Parallel  vector fields (or 1-forms) on Riemannian manifolds}

Let $(M,g)$ be a Riemannian manifold equipped with the Levi-Civita connection $\nabla$.

\begin{definition}
A vector field $X$ on $M$ is called \textit{parallel}  with respect to the Levi-Civita connection $\nabla$ if and only if
$$\nabla_YX=0, \quad \forall \, Y\in\mathfrak{X}(M).$$
For a local coordinate system $(x^i)$ on $M$, the vector field $X=X^i\frac{\partial}{\partial x^i}$ is parallel if and only if
$$X^i_{|j}=\frac{\partial X^i}{\partial x^j} +X^m G^i_{mj}=0.$$
\end{definition}

\begin{definition}\label{nul.Riem.} Let  ${R}$ be the curvature tensor of the Levi-Civita connection.
The \textnormal{nullity space} of  ${R}$ at a point $x\in M$ is the subspace of $T_xM$ defined by
$${\N}_{R}(x):=\{X\in T_xM| \,\,  {R}_x(X,Y)=0, \, \,\text{for all}\,\, Y\in T_xM\}.$$
The dimension of ${\N}_{R}(x)$, denoted by  ${\mu}_{{R}}(x)$, is the  \textnormal{nullity index} of ${R}$ at $x$.
If the  nullity index ${\mu}_{R}$ is constant,
then the map ${\mathcal{N}}_{R}:x\mapsto {\N}_{R}(x) $ defines a distribution $\N_{R}$ of
rank ${\mu}_{R}$, called the \textnormal{nullity distribution} of ${R}$.
Any  smooth section in the nullity distribution $\N_{R}$ is called  \textnormal{a nullity vector field}.
We denote by $\Gamma({\N}_{R})$ the $C^\infty(M)$-module of the nullity vector fields.

Locally a vector field $X=X^i\pa_i \in {\N}_{R}(x)$ if and only if
$$X^mR^h_{ijm}=0.$$
 \end{definition}

Similarly, we define the kernel  of the curvature $R$, as follows:
\begin{equation}\label{ker.Riem}
{Ker}_{R}(x):=\{X\in T_x(M)| \,\,  {R}_x(Y,Z)X=0, \, \,\text{for all}\,\, Y, Z\in T_xM\}.
\end{equation}
We denote by $\Gamma({Ker}_{R})$ the $C^\infty(M)$-module of the kernel vector fields.

We have the following lemma.
\begin{lemma}
The nullity space and kernel space of the Riemannian curvature coincide at each point $x\in M$.
\end{lemma}

\begin{proof}
By plugging    $X\in \Gamma({\N}_{R})$ into the following Bianchi's identity
$$R(X,Y)Z+R(Y,Z)X+R(Z,X)Y=0,$$
we get  $R(Y,Z)X=0$ which means that $X\in\Gamma({Ker}_{R})$ and hence
$${\N}_{R}\subset {Ker}_{R}.$$
Conversely, let $X\in\Gamma({Ker}_{R})$, then
$$g(R(Y,Z)X,W)=0,$$
where $g$ is the Riemannian metric.  But, by using the properties of the curvature $R$, we have
$$g(R(Y,Z)X,W)=g(R(X,W)Y,Z)=0, \,\,  \forall Z \in \mathfrak{X}(M). $$
But $g$ is non-degenerate, therefore, $R(X,W)Y=0, \, \forall \, Y \in \mathfrak{X}(M)$. Thus, $X\in\Gamma({\N}_{R})$ and ${Ker}_{R}\subset {\N}_{R}$. This completes the proof.
\end{proof}

\begin{proposition}
The set of  parallel vector fields  is a subspace of the nullity space at each point of $M$. That is, any parallel vector field $X=X^i\frac{\partial}{\partial x^i}$ necessarily satisfies $X^mR^h_{ijm}=0$.
\end{proposition}
\begin{proof}
It is enough to prove that each parallel vector field is a nullity vector.
Let $Z$ be a parallel vector field. Then, by using the definition of the curvature tensor
$$R(X,Y)Z=\nabla_X\nabla_YZ-\nabla_Y\nabla_XZ-\nabla_{[X,Y]}Z=0.$$
So $Z$ is a kernel vector and hence a nullity vector.
\end{proof}

One can ask,   is there any nullity vector field which is parallel?  A nullity vector field is not necessarily parallel. We have the following counter-example.
\begin{example}
Let $M=\{(x^1,x^2,x^3,x^4)\in \Real^4: x^2,x^3>0\}$. Consider the Riemannian metric
$$F=\sqrt{x^2x^3(y^1)^2+(y^2)^2+(y^3)^2+(y^4)^2}.$$
The geodesic spray coefficients are given by
 $$G^1=\frac{y^1(y^2x^3+x^2y^3)}{2x^2x^3},\quad G^2=-\frac{1}{4}x^3(y^1)^2,\quad G^3=-\frac{1}{4}x^2(y^1)^2,\quad G^4=0. $$
 Straightforward calculations lead to the non-zero coefficients of the Levi-Civita connection
 $$G^2_{11}=-\frac{1}{2}x^3,\quad G^3_{11}=-\frac{1}{2}x^2,\quad G^1_{12}=\frac{1}{2x^2},\quad G^1_{13}=\frac{1}{2x^3}. $$
 The non-zero components of the curvature tensor are given by
 $$R^1_{212}=-\frac{1}{4 (x^2)^2}, \quad R^1_{312}=\frac{1}{4 x^2x^3}, \quad R^2_{112}=\frac{x^3}{4 x^2}, \quad R^3_{112}=-\frac{1}{4},$$
  $$R^1_{213}=\frac{1}{4  x^2x^3}, \quad R^1_{313}=-\frac{1}{4 (x^3)^2}, \quad R^2_{113}=-\frac{ 1}{4 }, \quad R^3_{112}=\frac{x^2}{4 x^3}.$$
Solving the system $X^m R^h_{m jk}=0$, then any nullity vector field $X$ can be calculated on the form
 $$X=s\left(\frac{\partial}{\partial x^2}-\frac{\partial}{\partial x^3}\right)+t\frac{\partial}{\partial x^4}, \quad s, t \in \Real.$$
 So, the nullity space is spanned by
  $$\mathcal{N}_R(x)=Span\left\lbrace\frac{\partial}{\partial x^2}-\frac{\partial}{\partial x^3},\frac{\partial}{\partial x^4}\right\rbrace.$$
   Therefore the index of nullity is $2$.

 On the other hand, by solving the system $Y^i_{|j}=0$ then a parallel vector field $Y$ can be calculated as follows
 $$Y=t\frac{\partial}{\partial x^4},\quad t\in \Real$$
 and hence the space of parallel vector fields spanned by $\frac{\partial}{\partial x^4}$ and its dimension is $1$.
\end{example}

 Another example shows that the space of parallel vector fields is a proper subspace of the nullity space.

\begin{example}
Let
\begin{equation*}
 F=\frac{\sqrt{1-|a|^2}}{(1+\langle a,x\rangle)^2}\sqrt{|y|^2-\frac{2\langle a,y \rangle\langle x,y \rangle }{1+\langle a,x\rangle}-\frac{(1-|x|^2)\langle a,y \rangle^2 }{1+\langle a,x\rangle}},
\end{equation*}
where $y\in T_xB^n=\mathbb{R}^n$, $a=(a_1,a_2,...,a_n)\in \mathbb{R}^n$ is a constant vector with $|a|<1$, $|.|$ and $\langle .,.\rangle$ are the standard Euclidean norm and inner product in $\Real^n$.
 The spray coefficients are given by
 $$G^i=-\frac{\langle a,y \rangle}{1+\langle a,x \rangle}y^i. $$
 Straightforward calculations lead to the following.
$$G^i_j=-\frac{y^ia^\ell \delta_{\ell j}+\langle a,y\rangle\delta^i_j}{1+\langle a,x\rangle}, \quad
G^i_{jk}=-\frac{a^\ell \delta_{\ell j}\delta^i_k+a^\ell \delta_{\ell k}\delta^i_j}{1+\langle a,x\rangle},$$
$$\partial_{h}G^i_{jk}=\frac{(a^\ell \delta_{\ell j}\delta^i_k+a^\ell \delta_{\ell k}\delta^i_j)a^m \delta_{m h}}{(1+\langle a,x\rangle)^2}, \quad
{R}^{\,\,i}_{h\,\,jk}=0.$$

 By the help of the Finsler package \cite{CFG} and Maple program one can see that the vector $X=X^i\frac{\partial }{\partial x^i}$, where
 $$X^1=\frac{1}{a^1}(c^1(1+a^1+a^1 x^1)-\langle a,c\rangle)(1+\langle a,x\rangle), \quad X^\mu =(1+\langle a,x\rangle)(c^\mu+c^1x^\mu),  \,\mu=2,...,n$$
 is a parallel vector field, moreover, the space of parallel vectors has dimension  $n$.

Also the associated parallel one form is given by  $\beta=b_i(x)y^i$, where
$$b_1(x)=\frac{c+c_\mu x^\mu}{(1+\langle a,x\rangle)^2},\quad b_\mu(x)=\frac{a_\mu b_1}{a_1}-\frac{c_\mu(1+\langle a,x\rangle)}{a_1(1+\langle a,x\rangle)^2},$$
where  $\mu =2,...,n$.
\end{example}

\begin{remark}
Since the the Levi-Civita covariant derivative of the  Riemannian metric is zero, then one can see that there is a one-to-one correspondence between the parallel vector fields and the parallel one forms on a Riemannian manifold $M$.
\end{remark}

From now on, we use the notations $$\pa_i:=\frac{\partial }{\partial x^i}, \quad \paa_i:=\frac{\partial }{\partial y^i}.$$
 \begin{theorem}\label{Equiv_Riem.}
 Let $(M,\alpha)$ be a Riemannian manifold and its geodesic spray is $S$ and let $\beta=b_i(x)y^i$ be a one form on $M$  such that $b_i$ is gradient.  Then the following assertions are equivalent;

 \begin{description}
 \item[(a)] $\beta$ is parallel one form with respect to $\alpha$.

  \item[(b)]$d_S \beta=0$, $S$ the geodesic spray of $\alpha$.

 \item[(c)]$d_h \beta=0$, $h$ is the  horizontal projector associated with $S$.
\end{description}
\end{theorem}

\begin{proof}
The proof will be proceeded locally. Let $ \beta$ be a  one form on $M$  and $b_i$ is gradient.

\noindent $(a)\Longrightarrow (b)$ Let $ \beta$ be a parallel with respect to $\alpha$, i.e, $\partial_ib_j-G^r_{ij}b_r=0$. Then, we have
\begin{eqnarray*}
S\cdot \beta &=& y^i\partial_i\beta-2G^ib_i\\
&=&y^iy^j\partial_ib_j-y^iy^jG^r_{ij}b_r\\
&=&y^iy^j(\partial_ib_j-G^r_{ij}b_r)\\
&=&0 .
\end{eqnarray*}

\noindent $(b)\Longrightarrow (c)$ Assume that $d_S \beta=0$, then by taking into account the fact that $b_i$ is gradient, differentiating $y^iy^j(\partial_ib_j-G^r_{ij}b_r)$ with respect to $y^k$ yields $2y^j(\partial_kb_j-G^r_{kj}b_r)$. Hence, we have
\begin{eqnarray*}
d_h \beta &=&\partial_i\beta-N^r_{i}b_r\\
 &=&y^j\partial_ib_j-y^jG^r_{ij}b_r\\
&=&y^j(\partial_ib_j-G^r_{ij}b_r)\\
&=&0.
\end{eqnarray*}
\noindent $(c)\Longrightarrow (a)$ Suppose that $d_h \beta=0$.  Then, we get
\begin{eqnarray*}
d_h \beta = 0&\Longrightarrow& \partial_i\beta-N_i^rb_r=0\\
&\Longrightarrow &y^r\partial_ib_r-y^rG^j_{ir}b_j=0\\
&\Longrightarrow &y^r(\partial_ib_r-G^j_{ir}b_j)=0\\
&\Longrightarrow &\partial_ib_j-G^r_{ij}b_r=0,
\end{eqnarray*}
where we applied differentiation  w.r.t $y^j$. Hence, $\beta$ is parallel with respect to $\alpha$.
\end{proof}

\begin{corollary}
Let $(M,\alpha)$ be a Riemannian manifold and its geodesic spray is $S$ and let $\beta=b_i(x)y^i$ be a one form on $M$ with such that $b_i$ is gradient. The one form $\beta$ is parallel with respect to $\alpha$ if and only if it is a holonomy invariant function with respect to $S$.
\end{corollary}

\begin{proposition}\label{alpha_beta_indep.}
Let $(M,\alpha)$ be a Riemannian manifold  and  $\beta$ be a one form on $M$. Then the metrics $\alpha$ and $F=\alpha\phi(s), \ s:=\frac{\beta}{\alpha}$ are functionally independent, where $\phi$ is a positive, smooth and non-constant function on $\Real$.
\end{proposition}
\begin{proof}
Suppose that  the metrics $\alpha$ and $F=\phi(s)\alpha$ are functionally dependent. Then, the two form
$$dF\wedge d\alpha=\frac{\partial \phi}{\partial s}d\beta\wedge d \alpha$$
vanishes. Since $\phi$ is not constant then $\frac{\partial \phi}{\partial s}\neq 0$ and hence $d\beta\wedge d \alpha=0$. Now,
\begin{eqnarray*}
d\beta\wedge d \alpha = \pa_i \beta \ \pa_j \alpha \ dx^i\wedge dx^j
+  \paa_i \beta \ \paa_j \alpha \ dy^i\wedge dy^j
+\left(\pa_i \beta \ \paa_j \alpha-\pa_i \alpha \ \paa_j \beta \right) dx^i\wedge dy^j.
\end{eqnarray*}
Then, all of the   combinations in the right hand side must vanish, especially the combination $ \paa_i \beta \ \paa_j \alpha \ dy^i\wedge dy^j$  which vanishes only when $ \paa_i \beta \ \paa_j \alpha  $ is symmetric in $i$ and $j$, that is
 $$ \paa_i \beta \ \paa_j \alpha  -\paa_j \beta \ \paa_i \alpha =\ell_{i}b_j-\ell_jb_i=0, \quad \ell_i:=\paa_i \alpha.$$
Then contraction by $y^i$ yields   $\alpha b_j-\beta\ell_j=0$ and differentiation with respect to $y^k$ together with the property that $\ell_{i}b_j=\ell_jb_i$ gives
 $$\alpha \beta h_{jk}=0$$
  where $h_{jk}$ is the angular metric. Since none of $h_{jk}$, $\alpha$ and $\beta$ can  be zero we get a contradiction. Therefore the proof is completed.
\end{proof}

\begin{remark}\label{F_beta_metrics}
The above proposition is still valid if we replace the Riemannian metric $\alpha$ by a Finslerian one.
\end{remark}

\begin{proposition}\label{Riem_Objects}
  Any covariant symmetric tensor of type $(0,p)$ on a Riemannian manifold $(M,\alpha)$ being  parallel with respect to the Riemannian connection, in the sense that the covariant derivative of its components vanishes identically, induces a holonomy invariant function on $\T M$ with respect to the   geodesic spray  $S$. Moreover, the parallel anti-symmetric tensor on $M$ induces the zero function.
\end{proposition}
\begin{proof}
We prove the statement for a covariant tensor of type $(0,3)$ and then the proof of any such covariant tensors can be done in a similar manner. Let
$$T=T_{ijk}(x)dx^i\otimes dx^j \otimes dx^k$$
be a tensor of type $(0,3)$ on $M$ such that $T_{ijk;h}=0$, where the symbol "$;$" refers to the Riemannian covariant derivative.
Since  $T_{ijk;h}=0$, then we have
\begin{equation}\label{T_Covar_deriv}
\pa_h T_{ijk}-T_{\ell jk}G^\ell_{ih}-T_{i\ell k}G^\ell_{jh}-T_{ij\ell}G^\ell_{kh}=0.
\end{equation}
Now, define the function $$Q(x,y):=T_{ijk}y^iy^jy^k.$$
We claim that the $Q$ is holonomy invariant function on $\T M$. Indeed,   $Q$ is holonomy invariant with respect to $ S$  means that $d_hQ=0$ and locally gives $Q_{|i}=0$ where the symbol $|$ is the covariant derivative with respect to Berwald connection associated to $S$. Now, using \eqref{T_Covar_deriv} and the fact that $N^h_i=y^mG^h_{mi}$,  we have
\begin{align*}
Q_{|i}&=\pa_i Q-N^\ell_i\paa_\ell Q \\
      &=\pa_i (T_{jk\ell}y^jy^ky^\ell)-N^h_i\paa_h (T_{jk\ell}y^jy^ky^\ell) \\
      &=y^jy^ky^\ell\pa_i T_{jk\ell}-N^h_i T_{jk\ell}(\delta^j_h y^ky^\ell+\delta^k_h y^jy^\ell+\delta^\ell_h y^jy^k) \\
      &=y^jy^ky^\ell\pa_i T_{jk\ell}-y^m G^h_{mi} T_{jk\ell}(\delta^j_h y^ky^\ell+\delta^k_h y^jy^\ell+\delta^\ell_h y^jy^k) \\
      &=y^jy^ky^\ell\pa_i T_{jk\ell}-y^m G^h_{mi} T_{hk\ell} y^ky^\ell-y^m G^h_{mi}T_{jh\ell}  y^jy^\ell-y^m G^h_{mi}T_{jkh}  y^jy^k \\
      &=y^jy^ky^\ell\left( \pa_i T_{jk\ell}-  G^h_{ji} T_{hk\ell} - G^h_{ki}T_{jh\ell}   -  G^h_{\ell i}T_{jkh} \right)   \\
      &=y^jy^ky^\ell   T_{jk\ell;i}  \\
      &=0.
\end{align*}
That is, $Q$ is holonomy invariant with respect to $S$.

If, for example, $T_{ijk}=-T_{ikj}$, then we have $$T_{ijk}y^jy^k=T_{ikj}y^ky^j=-T_{ijk}y^ky^j$$
that is, $T_{ijk}y^jy^k=0$ and hence the proof is completed.
\end{proof}

 \begin{theorem}\label{Main_Theorem}
If a Riemannian manifold admits a parallel one form then the metrizability freedom of the geodesic spray is greater than one.    Or  equivalently, if the metrizability freedom of the geodesic spray of a Riemannian manifold is $1$ then the manifold does not admit a parallel one form.
\end{theorem}

\begin{proof}
Assume that  a Riemannian metric $\alpha$ and the coefficients of its geodesic spray $S$ are $G_\alpha^i$ with metrizability freedom $1$. Now, let $\beta$ a parallel one form with respect to $\alpha$. For a Finsler metric  $F$  of $(\alpha,\beta)$-type, then
$$G_F^i=G_\alpha^i+D^i,$$
where $G_F^i$ are the coefficients of the geodesic spray of $F$. It is known that the two sprays are equal if and only if $D^i$ vanishes. Moreover,  $D^i$ vanishes if and only if $\beta$ is parallel with respect to  $\alpha$.
Since we assumed that $\beta$ is parallel with respect to $\alpha$, then $D^i$ vanishes and $G_F^i=G_\alpha^i$.  This means that the spray $G_\alpha^i$ is the geodesic spray of $\alpha$ and $F$, but  by Proposition \ref{alpha_beta_indep.}, $dF\wedge d\alpha=d\beta\wedge d\alpha\neq 0$ which means that the freedom is greater than 1 and this is a contradiction.
  \end{proof}

It should be noted that the condition that the  metrizability freedom of the  geodesic spray is  greater than one is not sufficient for a Riemannian metric to admit a  parallel one form.
 This  can be shown by the following counter-example.
  \begin{example}
  Let $M=\mathbb{R}^4$ and consider the Riemannian metric
  $$\alpha=\sqrt{x^2(y^1)^2+x^1(y^2)^2+x^4(y^3)^2+x^3(y^4)^2}.$$
  Straightforward calculations lead to
$$G^1=\frac{1}{4} \frac{y^2(2y^1-y^2)}{x^2}, \quad G^2=\frac{1}{4} \frac{y^1(2y^2-y^1)}{x^1}, \quad G^3=\frac{1}{4} \frac{y^4(2y^3-y^4)}{x^4},\quad G^4=\frac{1}{4} \frac{y^3(2y^4-y^3)}{x^3}.$$
The non-zero  components $R^h_{ij}$ of the   curvature of the geodesic spray are given by
$$R^1_{12}=-\frac{1}{4} \frac{y^2(x^1+x^2)}{x^1(x^2)^2}, \quad R^2_{12}= \frac{1}{4} \frac{y^1(x^1+x^2)}{(x^1)^2x^2}, \quad R^3_{34}=-\frac{1}{4} \frac{y^4(x^3+x^4)}{x^3(x^4)^2}, \quad R^4_{34}= \frac{1}{4} \frac{y^3(x^3+x^4)}{(x^3)^2x^4}.$$
Assume that $\beta=b_i(x)y^i$ is parallel one form, then we must have $$R^h_{ij}b_h=0.$$
This yields the following two equations
$$R^1_{12}b_1+R^2_{12}b_2+R^3_{12}b_3+R^4_{12}b_4=0, \quad R^1_{34}b_1+R^2_{34}b_2+R^3_{34}b_3+R^4_{34}b_4=0. $$
By substituting and simplifying, we get
$$y^2x^1b_1-y^1x^2b_2=0, \quad y^4x^3b_3-y^3x^4b_4=0 .$$
It is clear that the above two equations are satisfied only if $b_i=0$ for all $i$. That is, there is no non-trivial parallel one form. To find the metrizability freedom, we have to calculate the following Lie brackets:
$$h_{12}:=[h_1,h_2]=-\frac{1}{4} \frac{y^2(x^1+x^2)}{x^1(x^2)^2} \frac{\partial}{\partial y^1}+\frac{1}{4} \frac{y^1(x^1+x^2)}{(x^1)^2x^2}\frac{\partial}{\partial y^1},$$
$$h_{34}:=[h_3,h_4]=-\frac{1}{4} \frac{y^4(x^3+x^4)}{x^3(x^4)^2} \frac{\partial}{\partial y^3}+ \frac{1}{4} \frac{y^3(x^3+x^4)}{(x^3)^2x^4}\frac{\partial}{\partial y^4},$$
$$h_{112}:=[h_1,h_{12}]=\frac{1}{8} \frac{y^2(x^1+3x^2)}{(x^1)^2(x^2)^2} \frac{\partial}{\partial y^1}-\frac{1}{8} \frac{y^1(x^1+3x^2)}{(x^1)^3x^2}\frac{\partial}{\partial y^1},$$
$$h_{334}:=[h_3,h_{34}]=\frac{1}{8} \frac{y^4(x^3+3x^4)}{(x^3)^2(x^4)^2} \frac{\partial}{\partial y^3}- \frac{1}{8} \frac{y^3(x^3+3x^4)}{(x^3)^3x^4}\frac{\partial}{\partial y^4},$$
where $h_i$ are the horizontal basis and one can see that the Lie brackets $h_{13}=h_{14}=h_{23}=h_{24}=0$. So we have only two linearly independent vectors out of the above vectors and the successive brackets do not generate new directions. That is the codimension of the holonomy distribution is $2$ and hence the metrizability freedom is $2$.

On the other hand,  by using Maple one can solve the system $d_hF=0$ and obtain the solution of the form
$$F(x,y)=f(x^2(y^1)^2+x^1(y^2)^2,x^4(y^3)^2+x^3(y^4)^2).$$
For example, the metric
$$F=\sqrt[4]{(x^2(y^1)^2+x^1(y^2)^2)^2+(x^4(y^3)^2+x^3(y^4)^2)^2}$$
is another Finsler metric has the same geodesic spray and this assures that the metrizability freedom is greater than $1$. In fact $F$ is Berwaldian.
  \end{example}

  We end this section by the following interesting result.
  \begin{theorem}
  A sufficient    condition  for a Riemannian manifold $(M,F)$ to admit a parallel one form  is
  $$  R^\mu _{hij}=0$$
  for some indices $\mu$.
  \end{theorem}
  \begin{proof}
  For a non-trivial  parallel one form $\beta=b_i(x)y^i$ on $M$, we have $d_h\beta=0$ and this implies the compatibility condition
  $$d_R\beta=0\Longrightarrow R^h_{ij}b_h=0.$$
  Now, if $R^\mu_{\ell ij}=0$ for some indices $\mu$, then  we have $b_\mu R^\mu_{\ell ij}=0$ and hence  the functions  $b_\mu(x)$  are arbitrary and the rest of the $b_i$'s are zero.  Then, we can choose the functions  $b_\mu(x)$   such that the   system
  $$\partial_k b_\mu-G^r_{k \mu}b_r=0$$
  is   satisfied. In fact, the condition $b_\mu R^\mu_{\ell ij}=0$ is the compatibility condition for the above system. 
     This completes the proof.
  \end{proof}

\section{Some applications}

  One of the interesting topics in Riemannian geometry is the existence of Killing vector fields.  It should be noted that the existence of a Killing vector field is related to the metrizability freedom of the geodesic spray of a Riemannian metric. According to the work of \cite{Duggal,Eisenhart}, the existence of a proper Killing vector field is equivalent to the existence of an essentially different second order tensor which has a constant covairant derivative. That is,  we have the following theorem.

 \begin{theorem}
If a Riemannian  manifold $(M,\alpha)$, $\alpha=\sqrt{a_{ij}y^iy^j}$, admits a proper affine Killing vector field then the metrizability  freedom of its geodesic spray   is greater than $1$.
 \end{theorem}

 \begin{proof}
 Assume that $(M,\alpha)$ is a Riemannian manifold with the geodesic spray $S$ admits a proper affine Killing vector field. Then, by  \cite{Duggal,Eisenhart}, there exists an essential different  second order tensor $k_{ij}$ with constant covariant derivative.  Then, an appropriate combination   $a(x)a_{ij}+b(x)k_{ij}$ produces   an essentially different solution $\overline{F}$ for the system $d_h\overline{F}=0$, for example, $\overline{F}=\sqrt{k_{ij}y^iy^j}$.
 That is, the metrizability freedom of geodesic spray is greater than $1$.
 \end{proof}

 By \cite{Szabo,Matveev},  each Berwald metric is geodesically equivalent a Riemannian metric. But not any Riemannian metric is geodesically equivalent to a Berwald metric. We have the following theorem.
\begin{theorem}
Let $(M,\alpha)$ be a Riemnnian metric with the geodesic spray $S$. Then, if $(M,\alpha)$ admits a parallel one form, then there exists a Berwald metric which is geodesically equivalent to $\alpha$.
\end{theorem}

\begin{proof}
Assume that $(M,\alpha)$ admits a parallel one form $\beta$, then the Finsler metric $F=\alpha+\beta$ is Berwald metric has the same geodesic spray $S$. That is $F$ geodesically equivalent to $\alpha$.
\end{proof}

\begin{theorem}
Let $M$ be a two dimensional Riemannian manifold  with the geodesic spray $S$ with non-vanishing curvature. Then  $M$ does not admit a parallel one form.
\end{theorem}
\begin{proof}
Let $M$ be a two dimensional Riemannian manifold  with the geodesic spray $S$ with non-vanishing curvature. Since the geodesic spray has non-vanishing curvature, then rank of the image of the curvature is $1$ and hence the rank of the holonomy distribution is $2+1=3$. Now the metrizability freedom of $S$ is equal to the codimension of the holonomy distribution which is $4-3=1$. That is the metrizability freedom is $1$, hence by Theorem \ref{Main_Theorem} there is no parallel one form.
\end{proof}

\section{Parallel one form  on Finsler manifold}

\begin{definition}
A one form $\beta=b_i(x)y^i$ a Finsler manifold $(M,F)$ is said to be horizontally parallel with respect to the induced Berwald connection if and only if $b_{i|j}=0$, where the symbol $|$ refers to the horizontal covariant derivative of Berwald connection. That is, a parallel one form $\beta$ satisfies the system
$$d_h\beta=0, \quad d_R\beta=0, \quad d_\C\beta=\beta.$$
\end{definition}

\begin{remark}
It should be noted that the horizontal covariant derivative of a scalar function with respect to Cartan, Berwald, Chern(Rund) and Hashiguchi connections coincide. However, we keep talking about the horizontal covairant derivative with respect to Berwald connection. Moreover, we ensure that there is no correspondence between parallel one form and parallel vector fields on Finsler manifold in contrast the Riemannian case.  This is because  the metric tensor is not horizontally constant w.r.t Berwald connection. This correspondence happens only  in the Landsberg spaces.
\end{remark}

 \begin{theorem}\label{Parallel_Finsler}
  Sufficient   conditions for a Finsler manifold $(M,F)$ to admit a parallel one form are
  $$R^\mu _{ij}=0, \quad G^\mu_{ijk}=0\ (\equiv G^\mu_{jk}=G^\mu_{jk}(x)).$$
  for some indices $\mu$.
  \end{theorem}
  \begin{proof}
  For a non-trivial  parallel one form $\beta=b_i(x)y^i$ on $M$, we have
$$\pa_ib_j-G^h_{ij}b_h=0.$$
Taking the derivative of the above equation with respect to $y^k$, we get a compatibility condition coming from the Berwald curvature, namely,  we have
$$G^h_{ijk}b_h=0.$$
Now, if $R^\mu_{\ell ij}=0$ and $G^\mu_{ijk}=0$ for some indices $\mu$, then  we have $b_\mu R^\mu_{\ell ij}=0$ and $b_\mu G^\mu_{ijk}=0$ and hence  the functions  $b_\mu(x)$  are arbitrary and the rest of the $b_i$'s are zero. Then, we can choose the functions  $b_\mu(x)$   such that the system
  $$\partial_k b_\mu-G^r_{k \mu}b_r=0$$
  is satisfied.  In fact, the conditions $b_\mu R^\mu_{\ell ij}=0$ and $b_\mu G^\mu_{ijk}=0$ are the compatibility conditions for the above system.  This completes the proof.
  \end{proof}

  Analogously to Theorem \ref{Main_Theorem}, we can see that if a Finsler manifold $(M,F)$ admits a parallel one form then the metrizability freedom of the geodesic spray is greater than $1$. Moreover, if the freedom is greater than $1$, then this does not imply the existence of a parallel one form. We give the following counter-example.

\begin{example}
Let $F$ be a projectively flat metric with zero flag curvature on $\mathbb{B}^2(1)$ studied by Shen \cite{Shen_paper}
\begin{equation*}
  F=\frac{(\sqrt{(1-|x|^2)|y|^2+\langle x,y\rangle^2}+\langle x,y\rangle)^2}{(1-|x|^2)^2\sqrt{(1-|x|^2)|y|^2+\langle x,y\rangle^2}}
\end{equation*}
with the geodesic spray given by the coefficients
\begin{equation*}
  G^i=\frac{\sqrt{(1-|x|^2)|y|^2+\langle x,y\rangle^2}+\langle x,y\rangle}{1-|x|^2} y^i.
\end{equation*}

 The parallel one form $\beta=b_i(x)y^i$ on $(\mathbb{B}^2(1),F)$ must satisfy $b_hG^h_{ijk}=0$,  which yields  the following single algebraic equation in $b_1$ and $b_2$ 
 $$  b_1(x) ((x^1)^2y^2-x^1x^2y^1-y^2)- b_2(x) ((x^2)^2y^1-x^1x^2y^2-y^1)=0.$$
But there are no non-trivial $b_i(x)$ such that the above equation is satisfied. That is, $(\mathbb{B}^2(1),F)$   does not admit a parallel one form although the curvature vanishes and the metrizability freedom is maximal.
\end{example}

   The following example provides a Finsler metric $F$ admitting a parallel one form.

\begin{example}

Let
\begin{equation*}
 F=\sqrt{\sqrt{(y^1)^4+x^1x^2(y^2)^4+(y^3)^4}+ (y^3)^2}.
\end{equation*}
 The spray coefficients are given by
 $$G^1=-\frac{1}{24}\frac{x^2(y^2)^4}{(y^1)^2}, \quad G^2=\frac{1}{24}\frac{y^2(3x^1y^2+4x^2y^1)}{x^1x^2},  \quad G^3=0. $$
 Straightforward calculations lead to the coefficients of  Berwald connection:
$$G^1_{11}=-\frac{1}{4}\frac{x^2(y^2)^4}{(y^1)^4},\quad G^1_{12}=\frac{1}{3}\frac{x^2(y^2)^3}{(y^1)^3},\quad G^2_{12}=\frac{1}{6x^1},\quad G^1_{22}=-\frac{1}{2}\frac{x^2(y^2)^2}{(y^1)^2},\quad G^2_{22}=\frac{1}{4x^2}.$$
The non-zero components $G^h_{ijk}$ of Berwald curvature are given as follows
$$G^1_{111}=  \frac{x^2(y^2)^4 }{ (y^1)^5}, \quad G^1_{112}=-  \frac{x^2(y^2)^3 }{ (y^1)^4},  \quad G^1_{122}=   \frac{x^2(y^2)^2 }{ (y^1)^3},  \quad G^1_{122}= -  \frac{x^2y^2 }{ (y^1)^2}.$$

The non-zero components $R^h_{jk}$ are given by
$$R^1_{12}=-\frac{1}{72} \frac{x^2(y^2)^3(x^1x^2(y^2)^4+10(y^1)^4)}{x^1(y^1)^6}, \quad R^2_{12}=\frac{1}{72}\frac{x^1x^2(y^2)^4+10(y^1)^4}{(x^1)^2(y^1)^3}.$$

Now, the one form $\beta=b_i(x)y^i$ is parallel, where
$$b_1(x)=0,\quad b_2(x)=0,\quad b_3(x)=Const.$$

It is to be noted  that in this example the metrizability freedom is greater than 1.
\end{example}

We end this work by the following theorem which works for both Finslerian and Riemannian manifolds.

 \begin{theorem}
 Let $(M,F)$ be a Finsler manifold. If one coefficient of the geodesic spray or more vanishes, then $(M,F)$ admits a parallel one form.
  \end{theorem}
  \begin{proof}
  Let $(M,F)$ be a Finsler manifold with the property that  $G^\mu=0$ for some indices $\mu$. The result follows from Theorem 5.3.  In more details, by the formula of $G^i_{jk}$ and \eqref{R12}, we have
  $$G^\mu_{ij}=0, \quad R^\mu_{ij}=0.$$
  Hence, the system
$$\pa_ib_\mu-G^\nu_{ij}b_\nu=0$$
becomes
$$\pa_ib_\mu=0.$$
 That is $b_\mu=c_\mu$ where $c_\mu$ are constant. Moreover the compatibility condition $R^\mu_{ij}b_\mu=0$ is satisfied. Consequently, the one form $\beta=b_\mu y^\mu=c_\mu y^\mu$ is parallel one form.
  \end{proof}

As an application of the above theorem see Examples 1 and 5.

 %%%%%%%%%%%%%%%%%%%%%%%%%%%%%%%%%%%%%%%%%%%%%%%%%%%%%%%%%%%%%%%%%%%%%%%%%%%%%%%%%%%%%%

 \providecommand{\bysame}{\leavevmode\hbox
to3em{\hrulefill}\thinspace}
\providecommand{\MR}{\relax\ifhmode\unskip\space\fi MR }
\providecommand{\MRhref}[2]{%
  \href{http://www.ams.org/mathscinet-getitem?mr=#1}{#2}
} \providecommand{\href}[2]{#2}

\end{document}